\documentclass{article}

\usepackage[english]{babel}

\usepackage[letterpaper,top=2cm,bottom=2cm,left=3cm,right=3cm,marginparwidth=1.75cm]{geometry}

\usepackage{amsmath}
\usepackage{amssymb}
\usepackage{amsthm}
\usepackage{graphicx}
\usepackage[colorlinks=true, allcolors=blue]{hyperref}

\theoremstyle{plain}
\newtheorem{theorem}{Theorem}
\newtheorem{lemma}[theorem]{Lemma}
\newtheorem{proposition}[theorem]{Proposition}

\theoremstyle{definition}

\title{Compact majority-minority districts almost never exist}
\author{Boris Alexeev \and Dustin G.\ Mixon\thanks{Department of Mathematics, The Ohio State University, Columbus, OH} \thanks{Translational Data Analytics Institute, The Ohio State University, Columbus, OH}}
\date{}

\begin{document}
\maketitle

\begin{abstract}
For a uniformly distributed population, we show that with high probability, any majority-minority voting district containing a fraction of the population necessarily exhibits a tiny Polsby--Popper score.
\end{abstract}

\section{Introduction}

Gerrymandering is the manipulation of voting district boundaries to favor a party or class.
One of the telltale signs of gerrymandering is a distorted shape, which can be quantified by some notion of geographic compactness.
For example, the \textbf{Polsby--Popper score} \cite{PolsbyP:91} of a district $D\subseteq\mathbb{R}^2$ is given by 
\[
\operatorname{PP}(D)
:=\frac{4\pi\cdot\operatorname{area}(D)}{\operatorname{length}(\partial D)^2},
\]
where $\partial D$ denotes the boundary of $D$.
(To ensure that $\operatorname{PP}(D)$ is well defined, we assume $D$ has rectifiable boundary.)
Inspired by the isoperimetric inequality~\cite{Osserman:78}, this scale-invariant score ranges from $0$ to $1$, with smaller scores being achieved by stranger shapes.

We say a district is \textbf{majority-minority} if the majority of its population belongs to a prescribed minority (be it political, racial, or otherwise).
In this paper, we show that under a simple model of population distribution, strange shapes are inevitable when drawing majority-minority districts.

\begin{theorem}
\label{thm.main result}
There exists a universal constant $c_0>0$ for which the following holds:
Given a state $S\subseteq\mathbb{R}^2$ with rectifiable boundary, put $\kappa:=\operatorname{PP}(S)$.
Let $\alpha\in(0,\frac{1}{2})$ denote the fraction of the population that belongs to a prescribed minority, and let $\beta\in(0,1)$ denote the fraction of the population that must reside in a given district.
If $n$ population locations are drawn independently and uniformly at random from $S$, then with probability\footnote{Here and throughout, $o_{\kappa,\alpha,\beta}(1)$ represents quantities that vanish when $\kappa$, $\alpha$, and $\beta$ are fixed and $n\to\infty$.} $1-o_{\kappa,\alpha,\beta}(1)$, it holds that every possible majority-minority district $D\subseteq S$ with rectifiable boundary that contains at least $\beta n$ of the population of $S$ necessarily satisfies
\[
\operatorname{PP}(D)
\leq c_0\cdot\frac{1}{1-2\alpha}\cdot\frac{\log^2n}{\sqrt{\beta n}}.
\]
\end{theorem}

As one might expect, Theorem~\ref{thm.main result} is stronger in cases where $\alpha$ is smaller and $\beta$ is larger, but even if $\alpha$ is close to $\frac{1}{2}$ and $\beta$ is close to $0$, the Polsby--Popper score of a majority-minority district is doomed to be small when $n$ is large.
To simplify its statement, our theorem requires districts to contain a fixed fraction $\beta$ of the population (regardless of $n$).
This model is appropriate, for example, in the case of U.S.\ Congressional districts, which must contain at least $\beta=\frac{1}{435}$ of a state's population (regardless of the size of the total population) as a consequence of federal law.
However, one can modify our proof to allow for much smaller populations in $D$, e.g., $\log^5n$.

This paper can be viewed as the third installment in a series.
First, in~\cite{AlexeevM:18}, the authors proved an impossibility theorem for gerrymandering that exposed a fundamental tension between one person--one vote (a federal requirement that the districts in a given state have populations of the same size), Polsby--Popper compactness (a common requirement that the Polsby--Popper score of a voting district not be too small), and bounded efficiency gap (a measure of gerrymandering introduced in~\cite{StephanopoulosM:15} that compares vote share and seat share).
In~\cite{AlexeevM:18}, we optimized the proof for length, though at the price of clarity.
Next, we participated in the 2023 \textit{Summer of Math Exposition} contest by submitting a follow-up video~\cite{WaitingForZeno:23}.
In contrast to our original paper, our discussion in the video was optimized for clarity, and as part of this optimization, we discovered a new and improved proof of our impossibility theorem.
This new proof revealed how our result applies not only to partisan gerrymandering (as quantified by the efficiency gap), but also to any gerrymandering that benefits a minority population.
The current paper further builds on this improvement by upgrading the main result from ``there exists an arrangement of the population that suffices'' to ``almost all arrangements of the population suffice''.

In the next section, we follow the approach in our video~\cite{WaitingForZeno:23} to prove a model-free version of our main result using ideas from planar geometry and combinatorics.
Section~3 then proves Theorem~\ref{thm.main result} by combining the model-free result with our probabilistic model of the population locations.
This part of the proof follows from several applications of the Chernoff bound.
We conclude in Section~4 with a brief discussion.

\section{Model-free result}

In this section, we consider an infinite grid of square \textbf{precincts} in the plane.
Given a state in the plane, the precincts whose interiors are contained in the state are called \textbf{inner precincts}, and the precincts whose interiors intersect the state boundary are called \textbf{boundary precincts}.
We also partition each precinct into a $k\times k$ grid of square \textbf{census blocks} for some fixed choice of $k\in\mathbb{N}$.
What follows is a model-free version of our main result.

\begin{theorem}
\label{thm.model free}
Fix $a,b,c>0$, and suppose majority and minority populations are distributed in a state $S\subseteq\mathbb{R}^2$ in such a way that
\begin{itemize}
\item[(i)]
the minority population of each precinct is at most $a$,
\item[(ii)]
the majority population of each inner precinct is at least $b$ more than its minority population, and
\item[(iii)]
the total population of each census block is at most $c$.
\end{itemize}
Then any majority-minority district $D\subseteq S$ with rectifiable boundary that contains a population of $d>4c$ satisfies
\[
\operatorname{PP}(D)\leq 8\pi\cdot\frac{a+b}{b}\cdot\frac{k}{\sqrt{d/c}-2}.
\]
\end{theorem}

Our proof of Theorem~\ref{thm.model free} makes use of several straightforward lemmas.

\begin{lemma}
\label{lem.urn}
Fix $a,b>0$, and consider $s$ urns, each containing
\begin{itemize}
\item[(i)] at most $a$ green balls, and
\item[(ii)] at least $b$ more orange balls than green balls.
\end{itemize}
Suppose we discard balls from $r$ of the urns in such a way that the total number of remaining green balls exceeds the total number of remaining orange balls. Then $r>\frac{b}{a+b}\cdot s$.
\end{lemma}

\begin{proof}
Each urn we leave alone incurs a deficit of at least $b$ balls, while each urn we ransack offers a surplus of at most $a$ balls.
If we end up with more green balls than orange balls, then
\[
ra
\geq\left(\begin{array}{c}\text{total surplus}\\\text{from ransacking}\end{array}\right)
>\left(\begin{array}{c}\text{total deficit}\\\text{from leaving alone}\end{array}\right)
\geq(s-r)b.
\qedhere
\]
\end{proof}

\begin{lemma}
\label{lem.squares per unit length}
Partition the plane into a grid of unit squares.
Then every curve of unit length intersects the interior of at most $4$ of these squares.
\end{lemma}

\begin{proof}
The smallest axis-aligned rectangle that contains the curve has length at most $1$ and width at most $1$.
It follows that the curve is contained in a unit square, which in turn intersects the interior of at most $4$ of the grid's unit squares.
\end{proof}

\begin{lemma}
\label{lem.buffon}
Given a closed curve $C\subseteq\mathbb{R}^2$ with $\operatorname{length}(C)>1$ and a finite point set $P\subseteq C$ such that every unit-length subsegment of $C$ intersects $P$ in at most $t$ points, it holds that $|P|\leq t\cdot\operatorname{length}(C)$.
\end{lemma}

\begin{proof}
Draw a unit-length subsegment $U\subseteq C$ uniformly at random.
Then $\frac{|P|}{\operatorname{length}(C)}=\mathbb{E}|U\cap P|\leq t$.
\end{proof}

\begin{lemma}
\label{lem.large set has large diameter}
Partition the plane into a grid of unit squares, and consider any point set $A\subseteq\mathbb{R}^2$ with at most $c$ points in each square.
Then every finite subset $B\subseteq A$ satisfies $\operatorname{diam}(B)\geq\sqrt{|B|/c}-2$.
\end{lemma}

\begin{proof}
By the pigeonhole principle, $B$ hits at least $|B|/c$ of the grid's squares.
Consider the smallest axis-aligned rectangle that contains these squares, and let $\ell$ and $w$ denote its length and width.
Then
\[
\frac{|B|}{c}
\leq\ell w
\leq\max(\ell,w)^2
\leq(\operatorname{diam}(B)+2)^2.
\qedhere
\]
\end{proof}

Before proving Theorem~\ref{thm.model free}, we first introduce two more pieces of terminology.
First, the precincts whose interiors intersect the district are called \textbf{constituent precincts}.
Second, we say a constituent precinct is \textbf{split} if the district boundary intersects the interior of the precinct. 

\begin{proof}[Proof of Theorem~\ref{thm.model free}]
Consider any majority-minority district $D\subseteq S$.
We will work at the scale in which the precincts have unit area, and so
\[
\operatorname{PP}(D)
\leq4\pi\cdot\frac{\#(\text{ constituent precincts })}{\operatorname{length}(\partial D)^2}.
\]
Identify each constituent precinct with an urn of green and orange balls that correspond to its minority and majority populations, respectively.
We will use Lemma~\ref{lem.urn} to establish that a fraction of the constituent precincts must be split.
Note that each constituent precinct that is also a boundary precinct of the state is necessarily split, and we can encode this requirement by adding a large number of orange balls to the corresponding urn.
Overall, Lemma~\ref{lem.urn} implies
\[
\operatorname{PP}(D)
\leq4\pi\cdot\frac{a+b}{b}\cdot\frac{\#(\text{ split precincts })}{\operatorname{length}(\partial D)^2}.
\]
Select $P\subseteq \partial D$ consisting of a single point from the interior of each split precinct.
By Lemma~\ref{lem.squares per unit length}, every unit-length subsegment of $\partial D$ intersects at most $4$ of the points in $P$, and assuming $\operatorname{length}(\partial D)>1$, Lemma~\ref{lem.buffon} then gives $\#(\text{ split precincts })=|P|\leq 4\cdot \operatorname{length}(\partial D)$. 
Thus,
\[
\operatorname{PP}(D)
\leq16\pi\cdot\frac{a+b}{b}\cdot\frac{1}{\operatorname{length}(\partial D)},
\]
provided $\operatorname{length}(\partial D)>1$.
This bound also holds when $\operatorname{length}(\partial D)\leq1$ since $\operatorname{PP}(D)\leq1$ by the isoperimetric inequality.
Finally, let $B$ denote the set of all $d$ population locations in $D$.
Lemma~\ref{lem.large set has large diameter} implies $\operatorname{length}(\partial D)\geq 2\cdot\operatorname{diam}(B)\geq2\cdot\frac{1}{k}\cdot (\sqrt{d/c}-2)$, which is positive by assumption.
(The extra factor of $\frac{1}{k}$ follows from our unit-precinct scaling, which makes the census blocks have side length $\frac{1}{k}$.)
Combining with our previous bound gives the result.
\end{proof}

\section{Proof of Theorem~\ref{thm.main result}}

In this section, we show that for population locations drawn independently and uniformly at random, the hypotheses in Theorem~\ref{thm.model free} are satisfied with high probability for appropriate choices of $k\in\mathbb{N}$ and $a,b,c>0$.
We establish this by repeatedly applying the Chernoff bound:

\begin{proposition}
Fix $\delta\in(0,1)$.
Then $X\sim\mathrm{Binomial}(n,p)$ satisfies
\[
(1-\delta)np
\leq X
\leq (1+\delta)np
\]
with probability at least $1-2e^{-\delta^2np/3}$.
\end{proposition}

First, we select the sizes of the state, the precincts, and the census blocks.
Without loss of generality, we scale $S$ to have
\[
\operatorname{area}(S)=\frac{n}{\log^4n},
\]
and we work with precincts of unit area.
Following the proof of Theorem~\ref{thm.model free}, we may combine Lemmas~\ref{lem.squares per unit length} and~\ref{lem.buffon} to see that the state has at most $4\operatorname{length}(\partial S)$ boundary precincts, provided $\operatorname{length}(\partial S)>1$.
Then the total number $m$ of precincts whose interior intersects the state satisfies
\[
m
\leq\operatorname{area}(S)+4\operatorname{length}(\partial S)
=\frac{n}{\log^4 n}+4\sqrt{\frac{4 \pi n}{\kappa\log^4 n}}
=(1+o_{\kappa}(1))\cdot \frac{n}{\log^4 n}.
\]
Letting $m_0$ denote the number of inner precincts in $S$, we have the bounds
\[
\underbrace{\operatorname{area}(S)-4\operatorname{length}(\partial S)}_{(1-o_{\kappa}(1))\cdot \frac{n}{\log^4 n}}
\leq m_0
\leq \underbrace{\operatorname{area}(S)}_{\frac{n}{\log^4 n}}
\leq m
\leq \underbrace{\operatorname{area}(S)+4\operatorname{length}(\partial S)}_{(1+o_{\kappa}(1))\cdot \frac{n}{\log^4 n}}.
\]
Finally, we take the census blocks to have edge length $\frac{1}{k}$, where
\[
k
:=\bigg\lfloor\sqrt{\frac{n}{12\operatorname{area}(S)\log n}}\bigg\rfloor
=(1-o(1))\cdot\sqrt{\frac{\log^3 n}{12}}.
\]

We may now identify choices for $a$, $b$, and $c$ that satisfy the hypotheses in Theorem~\ref{thm.model free} with high probability.
We start by identifying $c$.
The population $X_i$ in the $i$th census block has distribution $\operatorname{Binomial}(n,p_i)$ for some
\[
p_i
\leq p
:=\frac{1}{k^2\operatorname{area}(S)}
=(1+o(1))\cdot\frac{12\log n}{n}.
\]
Since $X_i$ is stochastically dominated by $X\sim\operatorname{Binomial}(n,p)$, the Chernoff bound gives
\[
\mathbb{P}\Big\{X_i>\frac{3}{2}np\Big\}
\leq \mathbb{P}\Big\{X>\frac{3}{2}np\Big\}
\leq 2e^{-np/12}
\leq\frac{2}{n},
\]
and then applying the union bound over $mk^2$ census blocks gives that
\[
\max_iX_i
\leq \frac{3}{2}np
=18\log n
=:c
\]
with probability at least $1-mk^2\cdot\frac{2}{n}=1-o_{\kappa}(1)$.

Next, we identify $a$.
The minority population $Y_i$ in the $i$th precinct has distribution $\operatorname{Binomial}(\alpha n,q_i)$ for some
\[
q_i
\leq q
:=\frac{1}{\operatorname{area}(S)}
=\frac{\log^4n}{n}.
\]
Since $Y_i$ is stochastically dominated by $Y\sim\operatorname{Binomial}(\alpha n,q)$, the Chernoff bound gives
\[
\mathbb{P}\Big\{Y_i>\Big(1+\frac{1}{\log n}\Big)\alpha nq\Big\}
\leq \mathbb{P}\Big\{Y>\Big(1+\frac{1}{\log n}\Big)\alpha nq\Big\}
\leq 2e^{-\alpha nq/(3\log^2n)}
=2e^{-(\alpha/3)\log^2n},
\]
and then applying the union bound over $m$ precincts gives that
\[
\max_i Y_i
\leq \Big(1+\frac{1}{\log n}\Big)\alpha nq
= \Big(1+\frac{1}{\log n}\Big)\alpha\log^4n
=:a
\]
with probability at least $1-m\cdot2e^{-(\alpha/3)\log^2n}=1-o_{\kappa,\alpha}(1)$.

Finally, we identify $b$.
The majority population $Z_i$ in the $i$th inner precinct has distribution $\operatorname{Binomial}((1-\alpha)n,q)$.
Applying the Chernoff bound and the fact that $\alpha<\frac{1}{2}$ gives
\[
\mathbb{P}\Big\{Z_i<\Big(1-\frac{1}{\log n}\Big)(1-\alpha) nq\Big\}
\leq 2e^{-(1-\alpha) nq/(3\log^2n)}
\leq2e^{-\frac{1}{6}\log^2n},
\]
and then applying the union bound over $m_0$ inner precincts gives that
\[
\min_i Z_i
\geq \Big(1-\frac{1}{\log n}\Big)(1-\alpha) nq
= \Big(1-\frac{1}{\log n}\Big)(1-\alpha)\log^4n
\]
with probability at least $1-m_0\cdot2e^{-\frac{1}{6}\log^2n}=1-o(1)$.
This suggests we take
\[
b
:=\Big(1-\frac{1}{\log n}\Big)(1-\alpha)\log^4n-\Big(1+\frac{1}{\log n}\Big)\alpha\log^4n
=(1-2\alpha)\log^4n-\log^3n.
\]

Theorem~\ref{thm.main result} then follows from taking a union bound, applying Theorem~\ref{thm.model free}, and then making straightforward simplifications to the resulting bound.

\section{Discussion}

We conclude with some comments on Theorem~\ref{thm.main result}.
First, our theorem is not terribly sensitive to the precise probabilistic model.
For example, the population distribution need not be uniform.
Second, while we wrote our theorem in terms of the Polsby--Popper score, we expect a version of this result to hold for other notions of geographic compactness as well.
At its core, our proof establishes that the boundary of a majority-minority district is forced to split a constant fraction of the precincts, thereby essentially acting as a space-filling curve.
Meanwhile, any proper notion of geographic compactness should be able to detect a district that resembles a fractal.

\section*{Acknowledgments}

DGM was supported in part by NSF DMS 2220304.


\begin{thebibliography}{WW}

\bibitem{AlexeevM:18}
B.\ Alexeev, D.\ G.\ Mixon,
An impossibility theorem for gerrymandering,
Amer.\ Math.\ Monthly 125 (2018) 878--884.

\bibitem{WaitingForZeno:23}
B.\ Alexeev, J.\ Jasper, D.\ G.\ Mixon,
Why You Want Voting Districts To Be Only 4\% Pretty \#SoME3,
Waiting for Zeno,
\url{youtube.com/watch?v=jQGiiPQ_ac0}.

\bibitem{Osserman:78}
R.\ Osserman,
The isoperimetric inequality,
Bull.\ Amer.\ Math.\ Soc.\ 84 (1978) 1182--1238. 

\bibitem{PolsbyP:91}
D.\ D.\ Polsby, R.\ D.\ Popper,
The third criterion:\ Compactness as a procedural safeguard against partisan gerrymandering,
Yale Law Policy Rev.\ 9 (1991) 301--353.

\bibitem{StephanopoulosM:15}
N.\ O.\ Stephanopoulos, E.\ M.\ McGhee,
Partisan gerrymandering and the efficiency gap,
Univ.\ Chic.\ Law Rev.\ 82 (2015) 831--900.


\end{thebibliography}
\end{document}